\documentclass[11pt,letterpaper]{article}

\usepackage{graphicx}
\usepackage{enumerate}
\usepackage{amsmath}
\usepackage[english]{babel}
\usepackage{amssymb}
\usepackage{subfigure}
\usepackage{epsfig}
\usepackage{pstricks,pst-node,pst-coil,pst-plot,pst-text}
\usepackage{ifpdf}

\newtheorem{theorem}{Theorem}

\newtheorem{proposition}[theorem]{Proposition}

\newlength{\saveparindent}
\def\bproof{\begin{rm}\protect\vspace{5pt}\noindent\textbf{Proof: }%
\addtolength{\parskip}{4pt}\setlength{\parindent}{0pt}}
\def\eproof{\end{rm}\addtolength{\parskip}{-4pt}%
\setlength{\parindent}{\saveparindent}}

\newcommand{\bprooff}[1]{\begin{rm}\protect\vspace{5pt}%
\noindent\textbf{Proof of #1: }\addtolength{\parskip}{4pt}%
\setlength{\parindent}{0pt}}

\newenvironment{proof}{\par\bproof}{\eproof\(\qed\) \par\medskip}

\newcommand{\qed}{\quad\mbox{\rule{7pt}{7pt}}}

\newcommand{\calC}{\mathcal{C}}
\newcommand{\mi}{\texttt{-}}
\newcommand{\pl}{\texttt{+}}

\title{Non-degenerated groundstates in the antiferromagnetic Ising model on triangulations}

\author{
  Andrea~Jim\'enez\thanks{
  Departamento de Ingenier\'ia Matem\'{a}tica, Universidad de Chile \&
  Department of Applied Mathematics, Charles University.
  Web: \texttt{www.dim.uchile.cl/$\sim$ajimenez}.
  Gratefully acknowledges the support of Mecesup via UCH0607 Project,
    CONICYT and the partial support of the Czech Research Grant MSM 0021620838.}}

\begin{document}

\maketitle
\begin{abstract}
A triangulation is an embedding of a graph into 
  a closed Riemann surface so that each face boundary 
  is a $3$-cycle of the graph.
In this work, groundstate degeneracy in the 
  antiferromagnetic Ising model on triangulations is studied.
We show that for every fixed closed Riemann surface $\Omega$, 
  there are vertex-increasing sequences of triangulations of $\Omega$
  with a non-degenerated groundstate. In particular, we exhibit 
  geometrically frustrated systems with a non-degenerated groundstate.
\end{abstract}

\section{Introduction}\label{sec1} 
The Ising model is one of the most studied models of interacting particles  
  in statistical physics. This model has been strongly linked to the
  study of discrete mathematics~\cite{loebl09,GalLoe}. 
In this sense, tools and techniques developed in 
  the discrete setting have shown to be very useful to deal with the solution of problems 
  related to the Ising model and vice versa.

Typically, to study the Ising model, particles 
  are located at the vertices of a graph and the type of interaction between them
  is determined by the existence and weight of edges in the graph.
In this notes, we explore the Ising model where
  particles and their interaction describe triangulations
  of closed Riemann surfaces with edge-weight equal to $\mi 1$.

A \emph{triangulation} of a closed Riemann surface $\Omega$,
  or simply a triangulation, is an embedding of a graph 
  in $\Omega$ so that each face boundary is a $3$-cycle of the graph.
Throughout this work, the closed Riemann surface $\Omega$ will be specified	
  just in needed cases.

Let's introduce the main ingredients of the Ising model.
  Given a triangulation $T$, a \emph{state} of the Ising model 
  on $T$ is a function $\sigma$ that assigns to each 
  vertex of $T$ a value from the set $\{\pl1, \mi1\}$. 
  In other words, $\sigma \in \{\pl 1, \mi 1\}^{|V(T)|}$,
  where $V(T)$ denotes the set of vertices of $T$. 
Values $\pl1$ and $\mi1$ are usually called \emph{spins}. 
  Then, we also say that a state on $T$ is a \emph{spin-assignment} on $T$.
For each state $\sigma$, the \emph{energy} or \emph{Hamiltonian} of the 
  Ising model on a triangulation $T$ is defined by $
H(\sigma) = \mi \sum_ {uv \, \in \, E(T)} J_{uv} \sigma_{u} \sigma_ {v},
$ where $E(T)$ denotes the set of edges of $T$ and for each $uv$ in $E(T)$ 
  the parameter $J_{uv}$ is called \emph{coupling constant}.
The main purpose of a coupling constant $J_{uv}$ is to specify the type 
  of interaction between vertices $u$ and $v$. In general, this constant 
  may vary from positive to negative depending on the characteristics 
  of the system to be studied. It can also be randomly chosen, like in 
  the case of spin-glasses.

The \emph{antiferromagnetic} variant of the Ising model takes coupling constant
  equal to~$\mi1$ for all edges of the graph. Then, given a triangulation $T$
  and a state $\sigma$, the energy of $\sigma$ in the antiferromagnetic
  Ising model is given by 
\begin{equation} \label{eq:energy}
H(\sigma) =  \sum_ {uv \, \in \, E(T)} \sigma_{u} \sigma_ {v}.
\end{equation}

Many mathematical problems naturally arise from the Ising model. One of them
  is the study of states that provide the minimum possible energy for the system. 
Those states are usually known as \emph{groundstates}. 
  Related to the study of groundstates, is the \emph{groundstate degeneracy}, 
  which by definition corresponds to the 
  number of different groundstates that a system supports. If the 
  groundstate degeneracy of a system is greater than two (i.e. more
  than one pair of groundstates exist), it is said that 
  the groundstate is \emph{degenerated}.
  Otherwise, it is called \emph{non-degenerated}.
The groundstate degeneracy has been vastly studied~\cite{LV,PhysRevB.16.4630}, 
  being of great physical interest,
  because (among others) it determines the entropy of the system, 
  and characterizing entropy's behaviour helps to 
  understand physical phenomena associated to order and stability
  of the system~\cite{citeulike:1289400}. 

Let $T$ be a triangulation and $\sigma \in \{\pl1, \mi1\}^{|V (T)|}$ be a state
  of the antiferromagnetic Ising model on $T$. 
It is said that $uv \in E(T)$ is frustrated by $\sigma$ or that
  $\sigma$ frustrates $uv \in E(T)$ if $\sigma_{u} = \sigma_{v}$. 
Observe that each state on $T$ frustrates
  at least one edge of each face boundary of the triangulation since
  each face boundary is $3$-cycle.

This feature (every state frustrates at least one edge of each face boundary) is known as
  geometrical frustration. The understanding of order and stability of 
  geometrically frustrated systems, is one of the main 
  questions that condensed matter physicists face to explain.
It is expected that systems which exhibit geometrical
  frustration lead to highly degenerated groundstates 
  with a non-zero entropy at zero temperature~(see~\cite{2006PhT59b24M}). 
  In other words, groundstate degeneracy
  in a geometrically frustrated system is typically exponentially 
  large as a function of 
  the number of vertices of the underlying graph.
Indeed, groundstate degeneracy of any plane triangulation 
  is exponential in the number of vertices, since
   groundstate degeneracy of plane 
  triangulations is twice the number of perfect matchings 
  of cubic bridgeless planar graphs~(see~\cite{EKKKN,JKL}). 

Surprisingly, in this work it is shown that there are triangulations
   of closed Riemann surfaces with an arbitrary number 
   of vertices and a non-degenerated antiferromagnetic groundstate.
More precisely, we establish the next result.

\begin{theorem}\label{theo:onesat}
Let $\Omega$ be a fixed closed Riemann surface with positive genus ($g>0$). 
  Then, for every $n>0$ there is a triangulation 
  $T$ of $\Omega$ with $n \leq |V(T)|$ so that $T$ has a 
  non-degenerated antiferromagnetic groundstate.
\end{theorem}

In particular, when the closed Riemann surface $\Omega$ is a torus,
we have the following.

\begin{theorem}\label{theo:onesat-toro}
For every $n>0$ there is a toroidal triangulation 
  $T$ with $n \leq |V(T)|$ so that $T$ has a non-degenerated 
  antiferromagnetic groundstate. 
\end{theorem}

\section{Non-degenerated groundstates in triangulations} \label{sec:toroidal}

\subsection{Preliminaries} \label{sec:preliminaries}

Throughout this work, it will be needed to consider triangulations of 
  closed Riemann surfaces with some removed faces (with holes) 
  so that each hole is circumscribed by a $3$-cycle. 
Triangulations of this type will be called punctured triangulations~(see 
  for example~Figure~\ref{fig:twoholess}; triangles in grey depict holes).   
In general, every term defined for triangulations is naturally adapted to
  punctured triangulations. However, there are some 
  facts that hold only for triangulations; 
  they will be properly specified. We now introduce definitions for both 
  triangulations and punctured triangulations. 

Let $T$ be a (punctured) triangulation.
Recall that every spin-assignment on $T$  
  frustrates at least one edge of each face boundary of~$T$.
Notice that a spin-assignment is a groundstate if 
  it has the smallest possible number of frustrated edges 
  (see equation~\ref{eq:energy}). 
A spin-assignment $\sigma$ on $T$ is said to be \emph{satisfying} 
  if $\sigma$ frustrates exactly one edge of each face boundary of $T$.
 
Let $T$ be a triangulation.
Obviously any satisfying spin-assignment on $T$ is a groundstate.
  The converse is true for plane triangulations~\cite{JKL}.
Nevertheless, the equivalence does not hold in general because
  a satisfying spin-assignment doesn't need to exist.
However, note that when satisfying spin-assignments exist, 
  then every groundstate corresponds to a satisfying spin-assignment.
The situation is more complicated if $T$ is a punctured triangulation, 
  since distinct satisfying  spin-assignments on $T$ may 
  provide distinct antiferromagnetic energy. In Figure~\ref{punctured}
  an example is shown.

\begin{figure}[h]
\centering
\subfigure[$H(\sigma_{a}) = -5$]
{
\ifpdf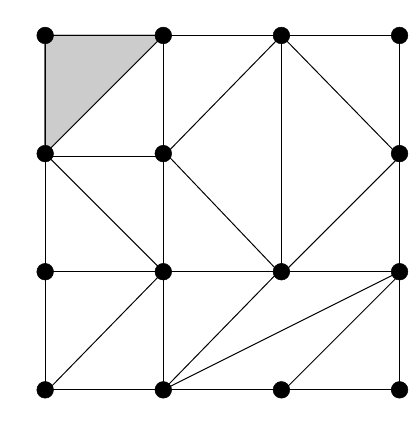\else\input{primero.pstex_t}\fi \label{primero}
}
\hspace{2em}
\subfigure[$H(\sigma_{b})=-6$]
  {\ifpdf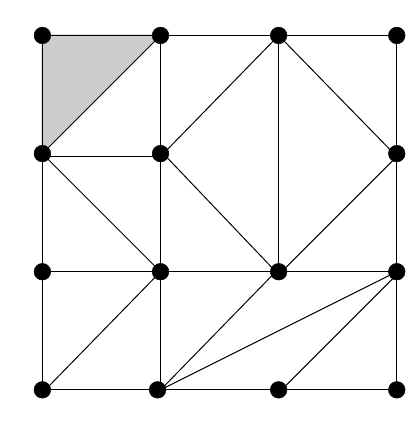\else\input{segundo.pstex_t}\fi \label{segundo}
  }
\caption{Consider the depicted punctured toroidal triangulation $T$ (triangle in grey represents a hole).
	In both pictures (a) and (b) a satisfying spin-assignment on $T$ is sketched; $\sigma_{a}$
	and $\sigma_{b}$ respectively. However, 
	the energy of $\sigma_{a}$ is greater
	than energy of $\sigma_{b}$.}  
\label{punctured}
\end{figure}

Observe that given a (punctured) triangulation $T$, a spin-assignment 
  $\sigma$ on $T$ is satisfying if and only if $\mi\sigma$ is 
  also a satisfying spin-assignment on $T$.
We shall refer to this fact as \emph{sign symmetry}. 
We will use it in order to reduce the number of cases
  that need to be analysed in the proofs.  
Moreover, if $T$ admits exactly 
  two satisfying spin-assignments $\sigma$ and $\mi\sigma$, 
  we say that $T$ admits a unique pair of satisfying spin-assignments.
  

If $\sigma$ assigns the same spin 
  on all vertices of a subgraph $H$ of $T$
  (respectively all elements of $S\subseteq V(T)$), 
  we say that $H$ (respectively a subset $S$)
  is \emph{monochromatic} under $\sigma$.   
Similarly, we say that an edge is monochromatic (respectively 
  non-monochromatic) under $\sigma$ if~$\sigma$ assigns the same 
  (respectively distinct) spins on both ends of the edge. In other
  words, an edge is monochromatic under $\sigma$ if and only if it is frustrated by $\sigma$.
Monochromatic and non-monochromatic faces are 
  defined analogously depending on whether or 
  not its face boundary is 
  either monochromatic or non-monochromatic. 
Then, a spin-assignment $\sigma$ to $T$ 
  is satisfying if and only if every face of $T$ 
  is non-monochromatic under $\sigma$.
An edge $e$ in $E(T)$ will be called \emph{serious}
  if and only if $e$ is monochromatic under 
  every satisfying spin-assignment on $T$.

In what follows, serious edges are depicted as thicker lines 
  and holes of punctured triangulations are depicted as grey areas.

\subsection{Non-degenerated groundstates in toroidal triangulations} \label{sec:toroidal}

This section is devoted to prove Theorem~\ref{theo:onesat-toro} and to discuss some
   features of toroidal triangulations with a non-degenerated groundstate.
Here, we show a strategy to construct toroidal triangulations with
    a non-degenerated groundstate, where a groundstate is a satisfying spin-assignment. 
Nevertheless, we strongly believe there are many other ways to construct 
    them and also that the class of toroidal triangulations 
    with a non-degenerated groundstate is not small.
The proof of Theorem~\ref{theo:onesat-toro} is constructive and 
    it is based on a simple idea. However, it is far from being trivial 
    to find the concrete triangulations to make that idea work.

Next, we give definitions and an overview of the proof.  
Let $T$ be a toroidal triangulation that admits exactly 
  one pair of satisfying spin-assignments (recall that when a 
  satisfying spin-assignment exists, satisfying spin-assignments are identical to groundstates) 
  and $\sigma$ denote a satisfying spin-assignment on $T$.
Let $\tilde{F}$ be a non-empty set of faces of $T$.  
  We say that the pair $(T,\sigma)$ is \emph{invariant under removal} 
  of $\tilde{F}$ if the punctured triangulation $\tilde{T}$ 
  obtained by removing all faces contained in the set $\tilde{F}$,
  has exactly one pair of satisfying spin-assignments. In particular, 
  $\tilde{\sigma}$ is a satisfying spin-assignment 
  on $\tilde{T}$ if and only $\tilde{\sigma} \in \{\pl\sigma , \mi \sigma\}$.
If such a set of faces $\tilde{F}$ exists, we say that $T$ is a
  \emph{supporting triangulation}, $\tilde{F}$ will be called a 
  \emph{removable set of faces of} $T$ and $\tilde{T}$ will be referred to as
  the \emph{supporting punctured triangulation} associated to $T$ and $\tilde{F}$; 
  when $T$ and $\tilde{F}$ are clear (or implicit) from the context we 
  just say that $\tilde{T}$ is a supporting punctured triangulation. 
  Clearly, $\tilde{T}$ is an embedding of a graph in a torus with 
  $|\tilde{F}|$ triangular holes. 
Observe that the existence of a removable set of faces 
  is not trivial: one could remove from $T$ a non-empty set of faces in such a way 
  that the obtained triangulation has more satisfying spin-assignments than~$T$.

Let $T$ be a supporting triangulation, $\tilde{F}$ be a removable
  set of faces of $T$, $\sigma$ be a satisfying spin-assignment 
  on $T$ and $\tilde{T}$ be the supporting punctured triangulation associated
  to $T$ and $\tilde{F}$. The face boundaries of the faces 
  in $\tilde{F}$ (in other words, the $3$-cycles circumscribing
  the holes of $\tilde{T}$), will be referred to as the \emph{expandable cycles} 
  of $\tilde{T}$. The set of edges contained in the expandable cycles
  which are monochromatic under $\sigma$ will be called the 
  \emph{fundamental edges} of $\tilde{T}$. Notice that each
  expandable cycle contains exactly one fundamental edge, since  
  each expandable cycle is non-monochromatic under $\sigma$.
  Clearly, fundamental edges of $\tilde{T}$ are serious.

Assume that a supporting triangulation $T$ exists. Then,
  we know that a supporting punctured triangulation $\tilde{T}$ associated
  to $T$ admits exactly one pair of satisfying spin-assignments, say $\sigma$,
  $\mi \sigma$ and all expandable cycles of $\tilde{T}$ 
  are non-monochromatic under $\sigma$ with its fundamental edges 
  monochromatic under $\sigma$.
In our construction, each hole of $\tilde{T}$ (circumscribed by an expandable
  cycle) will be covered by a plane triangulation in such a way 
  that the number of vertices of the toroidal triangulation 
  increases and the number of satisfying spin-assignments keeps constant. 
To achieve this task, it is needed a plane triangulation $\Delta$ with an arbitrary number
  of vertices and satisfying the following condition: if $\{xyz\}$ is the 
  face boundary of the outer face of $\Delta$, then there is a unique pair of 
  satisfying spin-assignments on $\Delta$ such that at least 
  one edge from the set of edges $\{xy, yz, xz\}$ is monochromatic; 
  say $xy$.
We refer to $\Delta$ as an \emph{augmenting triangulation}
  and we say that edge $xy$ is a fundamental edge of $\Delta$.
An augmenting triangulation always exists, since 
  plane triangulations are duals of cubic bridgeless planar graphs
  and every perfect matching of a cubic bridgeless planar graph $G$ 
  corresponds to the set of monochromatic edges under a 
  satisfying spin-assignment on $G^*$. Then, 
  an augmenting triangulation is the dual of a 
  cubic bridgeless planar graph with a 
  specified edge contained in exactly one 
  perfect matching.
  
Next, we show that supporting triangulations exist.
In Subsection~\ref{augmenting} we will describe a family of 
  augmenting triangulations. Finally, in Subsection~\ref{subsec:prooftheotoro} we
   formalize the proof of Theorem~\ref{theo:onesat-toro}.

\subsubsection{Supporting triangulations} 

The minimal triangulations of a surface are those that
  have every edge in a \emph{noncontractible} 
  $3$-cycle.
A \emph{splitting} of a vertex $v$ replaces the vertex $v$ by two vertices
  $v_1$ and $v_2$ connected by a new edge $v_1v_2$, 
  and replaces each edge $vu$ incident to $v$ either by the edge $v_1u$ or by $v_2u$. 
It is well-known that every triangulation of a given surface $\Omega$,
  may be generated by a sequence of vertex-splittings 
  from a minimal triangulation of $\Omega$. In general, 
  the set of minimal triangulations is finite for 
  every fixed surface~\cite{ springerlink:10.1007/BF02767355, springerlink:10.1007/BF02764905}. In particular,
  there are 21 minimal toroidal triangulations (see~\cite[\S 5.4]{MT}).

It is natural and potentially useful to look for supporting 
  triangulations in the set of minimal toroidal triangulations. 
On one hand its reduced number of vertices allow to verify 
  the required properties manually, on the other hand, 
  it strongly indicates a possible way to describe the 
  whole set of toroidal triangulations with a 
  non-degenerated groundstate since every toroidal
  triangulation may be generated from a minimal one. 
\begin{figure}[h]
\centering
{
\ifpdf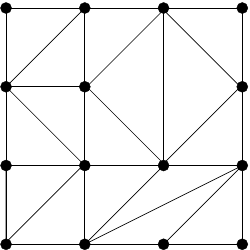\else\input{supp_01.pstex_t}\fi
}
\hspace{2em}
  {\ifpdf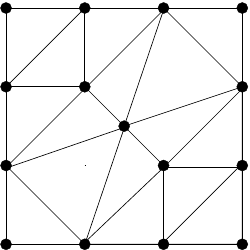\else\input{supp_02.pstex_t}\fi
  }
\hspace{2em}
{
\ifpdf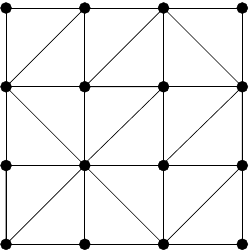\else\input{supp_03.pstex_t}\fi
}
\hspace{2em}
  {\ifpdf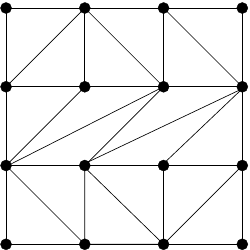\else\input{supp_04.pstex_t}\fi
  }
\vspace{2em}
{
\ifpdf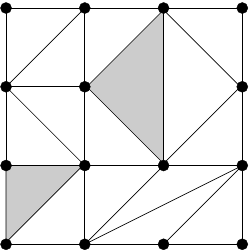\else\input{supp_01_removed.pstex_t}\fi
  \label{fig:supp_removed}
}
\hspace{2em}
  {\ifpdf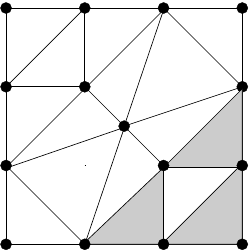\else\input{supp_02_removed.pstex_t}\fi
  }
\hspace{2em}
{
\ifpdf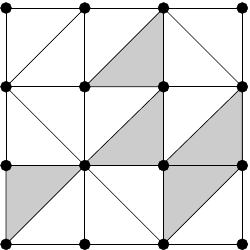\else\input{supp_03_removed.pstex_t}\fi
}
\hspace{2em}
  {\ifpdf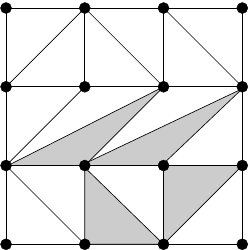\else\input{supp_04_removed.pstex_t}\fi
  }
\caption{In the first row, minimal triangulation of the torus which are supporting.
Below each supporting triangulation an associated supporting punctured triangulation 
  is depicted (regions in grey represent the holes generated by the removal
  of a removable set of faces of each supporting triangulation).}
\label{fig:supptriang}
\end{figure}

We exhibit the subset of minimal triangulations 
  of the torus which have the property of being a supporting
  triangulation (see Figure~\ref{fig:supptriang}). However, 
  some minimal triangulations of the torus which are not supporting can 
  become supporting after some slight modifications, 
  namely, after	edge-flippings and vertex-splittings. 
Nevertheless, this analysis is not the aim of 
  this work and will be studied separately. 

In the first row of Figure~\ref{fig:supptriang}, 
  all minimal toroidal triangulations which are supporting 
  triangulations are depicted  (in each picture
  opposite sides have to be identified). In the second 
  row of Figure~\ref{fig:supptriang} are depicted the supporting
  punctured triangulations obtained from the removal of 
  a set of removable faces of each supporting triangulation above.
We will formally prove supportability only for one of the 
  triangulation depicted in Figure~\ref{fig:supptriang}. 
  The same proof ideas can be easily applied for the remaining cases.

\begin{proposition}\label{prop:supptriang}
Let $T$ be the toroidal triangulation depicted in Figure~\ref{fig:noholess} 
and $\tilde{F}$ be the subset of faces of $T$ 
  with face boundary the cycles $\{uvw\}$ and $\{\tilde{u}\tilde{v}\tilde{w}\}$. 
Then, triangulation $T$ is a supporting triangulation and
$\tilde{F}$ is a removable set of faces of $T$.
\end{proposition}

\begin{figure}[h]
\centering
\subfigure[Toroidal triangulation.]
{
\ifpdf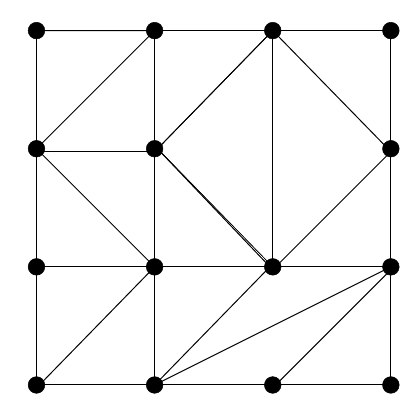\else\input{nohole_oness.pstex_t}\fi
  \label{fig:noholess}
}
\hspace{2em}
\subfigure[Punctured toroidal triangulation.]
  {\ifpdf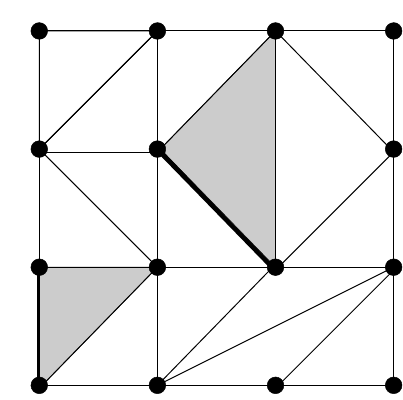\else\input{twohole_oness.pstex_t}\fi
  \label{fig:twoholess}
  }
\caption{The supporting triangulation and punctured triangulation of Proposition~\ref{prop:supptriang}.}
\end{figure}

\begin{proof}
Let $\tilde{T}$ be the punctured triangulation obtained by removing from $T$ all faces 
  contained in $\tilde{F}$ (see Figure~\ref{fig:twoholess}).
To see that $T$ is a supporting triangulation and $\tilde{F}$ is a
removable set of faces of $T$, it is enough to prove 
that $\tilde{T}$ has exactly one pair of satisfying spin-assignments and 
that both expandable cycles of $\tilde{T}$ are non-monochromatic under 
a satisfying spin-assignment~on~$\tilde{T}$.

To do that, by sign symmetry it suffices to verify that
  exactly one of the following three initial configurations can be 
  extended to a satisfying spin-assignment on $\tilde{T}$ (vertex labels as in Fi\-gure~\ref{fig:twoholess}):
  [1] when the cycle $\{uv\tilde{w}\}$ is assigned spin $\texttt{++-}$, 
  [2] when the cycle $\{uv\tilde{w}\}$ is assigned spin $\texttt{+-+}$, 
  [3] when the cycle $\{uv\tilde{w}\}$ is assigned spin $\texttt{+--}$. 
Then, it is necessary to check that both expandable 
  cycles are non-monochromatic under such satisfying spin-assignment.

In Figures~\ref{fig:nomono_no1} and \ref{fig:nomono_no2} the case when $uv$ is non-monochromatic
  is worked out; a subindex $i$  accompanying a $\pl$ or $\mi$ sign indicates that the spin is
  forced by the spin-assignments with smaller indices in order
  for the assignment to be satisfiable --- if spins assigned
  on the vertices of a face boundary are forced to be all of the same
  sign, then no satisfying assignment can exist under the given
  initial conditions.  This establishes that when the cycle 
  $\{uv\tilde{w}\}$ is assigned $\texttt{+-+}$ or 
  $\texttt{+--}$, there is no satisfying spin-assignment extension on $\tilde{T}$.
\begin{figure}[h]
\centering
\subfigure[Assignment forced by fixing $uv\tilde{w}$ to $\texttt{+-+}$.]
{
\ifpdf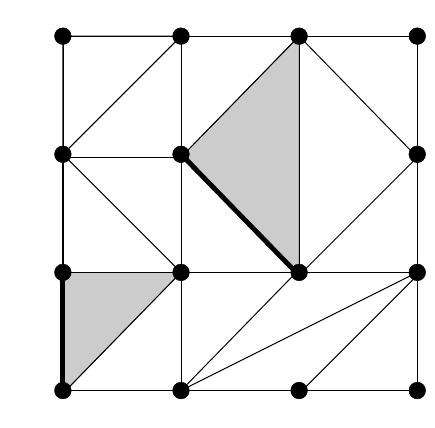\else\input{conector-a.pstex_t}\fi
  \label{fig:nomono_no1}
}
\hspace{2em}
\subfigure[Assignment forced by fixing $uv\tilde{w}$ to $\texttt{+--}$.]
  {\ifpdf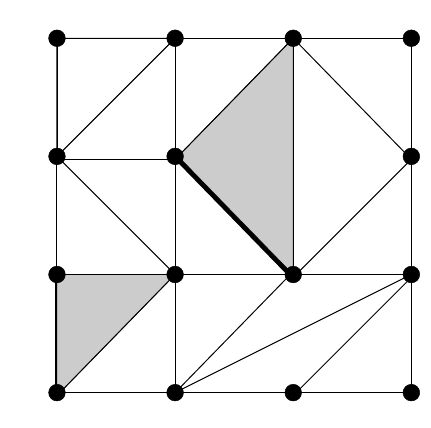\else\input{conector-b.pstex_t}\fi
  \label{fig:nomono_no2}
  }
\caption{Case when $uv$ is non-monochromatic. Forced monochromatic faces are labelled by~$\rightarrow\leftarrow$.}
\end{figure}

In Figure~\ref{fig:monos_supp}, the case when edge $uv$ is monochromatic
  is studied. In this situation, two subcases arise, 
  depending on whether or not the spin $\pl1$
  is assigned to vertex $x$ (vertex labels as in Figure~\ref{fig:twoholess}) 
  --- each subcase is dealt in the same way as the 
  previous situation and worked out separately in Figures~\ref{fig:mono_no} 
  and~\ref{fig:mono_si}. This shows that there exist a unique 
  satisfying spin-assignment on $\tilde{T}$
. 
\begin{figure}[h]
\centering
\subfigure[Sub-case when $\texttt{+}$ is pres\-cribed to vertex $x$. Initial 
spin-assignment forces a monochromatic triangle, labelled by $\rightarrow\leftarrow$.]
{\ifpdf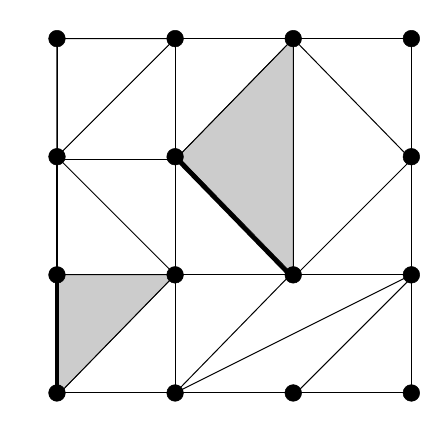\else\input{conector-c.pstex_t}\fi
  \label{fig:mono_no}
}
\hspace{2em}
\subfigure[Sub-case when $\texttt{-}$ is pres\-cribed to vertex $x$.
 Initial 
spin-assignment can be extended to a unique satisfying spin-assignment on $\tilde{T}$.]
  {\ifpdf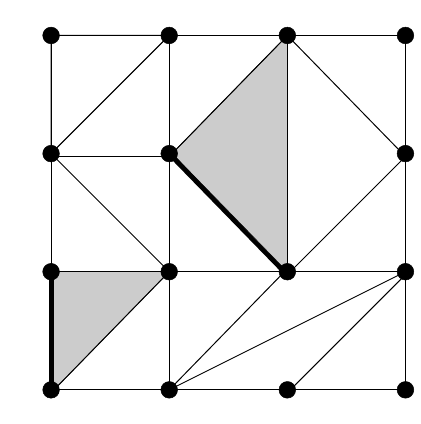\else\input{conector-d.pstex_t}\fi
  \label{fig:mono_si}
  }
\caption{The case when $\{uv\tilde{w}\}$ is prescribed spin $\texttt{++-}$.}
\label{fig:monos_supp}
\end{figure}

Finally, notice that both expandable cycles of $\tilde{T}$ are non-monochromatic under 
  the unique satisfying spin-assignment on $\tilde{T}$ (depicted in Figure~\ref{fig:mono_si}). 
\end{proof}

\subsubsection{Augmenting triangulations}\label{augmenting}

We already mentioned that an augmenting triangulation always exists.
In this subsection we just show an easy  
  way to construct a vertex-increasing family of such triangulations. 

Let $\Delta_0=\{x,y_0,z\}$ be a plane triangle. 
For $i \geq 1$, let $\Delta_{i}$ be the plane triangulation obtained 
  by applying the following rule to~$\Delta_{i-1}$:
  (1) insert a new vertex $y_i$ in the outer face of $\Delta_{i-1}$, and
  (2) connect the new vertex $y_i$ to 
      each vertex in the outer face of~$\Delta_{i-1}$ 
      so that the outer face of the new plane triangulation 
      has face boundary $\{xy_iz\}$ (See Figure~\ref{fig:stack}).
Clearly, the number of vertices of $\Delta_{n}$ is $n+3$.
We will see that every triangulation from the collection 
  $\{\Delta_{j}\}_{j>0}$ is an augmenting triangulation
  with fundamental edge $xy_j$ . 
This type of triangulations belongs to the set of 
\emph{stack triangulations}~(see~\cite{JK}) and more
  families of augmenting triangulations can be easily found in that set.

\begin{figure}[h]
\centering
\ifpdf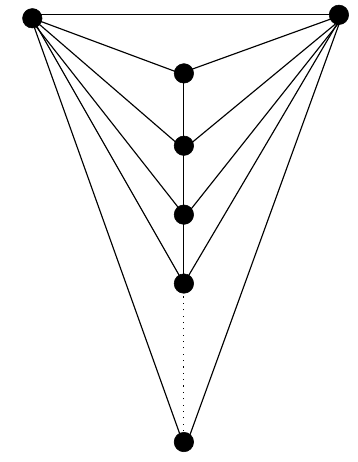\else\input{stack.pstex_t}\fi
\caption{Construction of $\Delta_n$.}
\label{fig:stack}
\end{figure}

\begin{theorem}
Let $n>0$ and $\{xy_nz\}$ the face boundary of the outer face of 
  $\Delta_n$ (see Figure~\ref{fig:stack}). There exist a unique pair of 
 satisfying spin-assignments on $\Delta_n$ so that edge $xy_n$ is monochromatic.
\end{theorem}
\begin{proof}
By sign symmetry, it suffices to prove that if spin  $\texttt{++-}$ is
    prescribed to $\{xy_nz\}$, then there
    is a unique satisfying spin-assignment extension on $\Delta_n$.
We will proceed by induction on~$n$. If $n=1$ it is trivial 
    to check that the result holds.
    Let $n>1$. If we prescribe spin $\texttt{++-}$ to $\{xy_{n}z\}$, then 
    in order to have a satisfying spin-assignment on $\Delta_n$
    the vertex $y_{n-1}$ is forced to have spin  $\texttt{-}$. 
    It implies that in any extension to a satisfying spin-assignment on $\Delta_n$, 
    the $3$-cycle $\{x,y_{n-1},z\}$ has spin $\texttt{+--}$. Then, by induction
    hypothesis uniqueness holds.  
\end{proof}

\subsubsection{Proof of Theorem~\ref{theo:onesat-toro}} \label{subsec:prooftheotoro}

Let $T$ be a supporting triangulation, $\tilde{F}$ be a set of removable
 faces of $T$ and $\tilde{T}$ be the supporting punctured triangulation 
 associated to $T$ and $\tilde{F}$. Consider $|\tilde{F}|=t$ and let 
 $\{\Delta^{i}\}_{i \in [t]}$ by a collection of augmenting triangulations. 
Let $\calC_1, \ldots, \calC_t$ denote the expandable cycles of $\tilde{T}$

Let $\mathcal{T}$ denote the triangulation obtained by gluing together the
    supporting punctured triangulation $\tilde{T}$ and the collection
    of augmenting triangulations $\{\Delta^{i}\}_{i \in [t]}$ in the following way: take 
    the collection of augmenting triangulations 
    $\{\Delta^{i}\}_{i \in [t]}$ and
    for each $i \in [t]$, identify the boundary face of 
    the outer face of $\Delta^{i}$ with the expandable cycle 
    $\calC_i$ of $\tilde{T}$ in such a way 
    that the fundamental edge the augmenting triangulation 
    coincides with the fundamental edge of $\calC_i$.

It follows from the construction that the toroidal triangulation $\mathcal{T}$ has a unique pair of groundstates.
This finishes the proof of Theorem \ref{theo:onesat-toro}.

\subsection{Proof of Theorem~\ref{theo:onesat}} \label{sec:generalsurface}

The proof of Theorem~\ref{theo:onesat} is relatively straightforward
  from the construction made for proving Theorem~\ref{theo:onesat-toro}.
Unlike the case of triangulations of the torus, 
 in Theorem~\ref{theo:onesat} the genus of the closed Riemann surface may be arbitrarily large.
However, the described supporting punctured triangulations can
  perform the task of increasing genus and at the same time keeping
  all properties of existence and uniqueness of satisfying
  spin-assignments in such a way that the same
  construction made in Subsection~\ref{subsec:prooftheotoro} works. 
  Next, we add some new definitions
  and show how to deal with the construction for proving Theorem~\ref{theo:onesat}.

Let $T$ be a supporting triangulation and $\tilde{F}$ be a set of removable
 faces of $T$. If $\tilde{F}$ contains at least two faces $f_1$, $f_2$ such that
 its face boundaries $\calC_1$, $\calC_2$ don't share any vertex, we say that 
 the supporting punctured triangulation $\tilde{T}$ associated to $T$ and $\tilde{F}$ 
 is a \emph{connector} and that $\calC_1$ and $\calC_2$ are its \emph{connection cycles}.
Clearly, a connector exists. Indeed, the supporting punctured triangulation depicted
 in Figure~\ref{fig:twoholess} is a connector. Properties and names from 
  supporting punctured triangulations are transferred
  to connectors.

We are ready to describe the construction. Let $g$ be a integer positive number, 
  $\{\tilde{T}_{i}\}_{i \in [g]}$ be a collection of connectors and $t_i+2$ be the number of 
  expandable cycles of connector $\tilde{T}_{i}$ for each~$i \in [g]$ (clearly, $t_i\geq0$ for all $i$). 
Moreover, let $\calC_{2i-1}$, $\calC_{2i}$ denote the connection cycles of $\tilde{T}_{i}$
for each $i \in [g]$. 

Let $\mathcal{\tilde{T}}$ denote the punctured triangulation 
    of a surface of genus $g$ obtained by gluing together 
    the collection of connectors $\{\tilde{T}_{i}\}_{i \in [g]}$
    in the following way: for each $j\in[g-1]$ identify the connection cycle
    $\calC_{2j}$ of $\tilde{T}_{j}$ with the connection cycle $\calC_{2(j+1)-1}$
    of $\tilde{T}_{j+1}$ in such a way that the fundamental edge of $\calC_{2j}$ 
    coincides with the fundamental edge of $\calC_{2(j+1)-1}$.

It is routine to check that $\mathcal{\tilde{T}}$
  has exactly one pair of satisfying spin-assignments and that each hole
  of $\mathcal{\tilde{T}}$ is non-monochromatic under any satisfying state on $\mathcal{\tilde{T}}$.

To conclude, the Theorem~\ref{theo:onesat} follow directly from applying the same 
  construction presented in Subsection~\ref{subsec:prooftheotoro}, using
  $\mathcal{\tilde{T}}$ instead of a supporting punctured triangulation.

\section{Final Comments}

The strategy to construct triangulations presented in this notes 
  may be (easily) extended to construct triangulations of a fixed surface 
  with $n$ vertices and groundstate degeneracy $f(n)$, where $f(n)$
  is a function depending on $n$ which can be either constant or polynomial on $n$;
  it can be reached by taking instead of an augmenting triangulation, a plane
  triangulation with the required property (which always exists --- see for example~\cite{JK}).

We believe that this work leaves many doors open and unanswered questions.
First, we think that it is of particular relevance to find a 
  complete description of triangulations 
  with a non-degenerated groundstate. 
Also, in this context, we strongly believe that it is of great interest 
  to study the following question: what is the groundstate 
  degeneracy of random triangulations 
  provided with the antiferromagnetic Ising model?. 
  This will help to understand the behaviour of geometrically frustrated systems.

On the other hand, the problem of spin glasses has attracted 
  considerable attention over recent years. 
  Both, in solid physics and in statistical physics (for instance see~\cite{1996JPhA...29.3939R}). 
  In the Ising spin glass model, coupling constants are randomly distributed. 
  Usually, each coupling constant is set randomly to 
  either $\pl1$ or $\mi1$ with equal probability. The case of the antiferromagnetic
  Ising model is the critical case when the coupling constant 
  is set to $\mi1$ with probability equal to one.
In this context, next goal would be to study the Ising model 
  on toroidal triangulations with this probability less than one and very close
  to one.

\section*{Acknowledgements}
The author thanks Marcos Kiwi and Martin Loebl for their enthusiasm, motivation, and 
  many helpful discussions.

\end{document}